\documentclass[12pt]{amsart}
\usepackage{amsmath,amssymb,
amsbsy,amsfonts,latexsym,amsopn,
amstext,cite,amsxtra,euscript,amscd,bm,mathabx}
\usepackage{url}

\usepackage[hmargin=3cm,vmargin=3cm]{geometry}

\usepackage[colorlinks,linkcolor=blue,
anchorcolor=blue,citecolor=blue,backref=page]{hyperref}
\usepackage{color}
\usepackage{graphics,epsfig}
\usepackage{graphicx}
\usepackage{float}
\usepackage{epstopdf}
\hypersetup{breaklinks=true}

\everymath{\displaystyle}

\usepackage[norefs,nocites]{refcheck}

\newtheorem{theorem}{Theorem}[section]
\newtheorem{lemma}[theorem]{Lemma}

\newcommand{\R}{\mathbb{R}}

\newcommand{\N}{\mathbb{N}}
\newcommand{\Z}{\mathbb{Z}}

\newcommand{\cP}{\mathcal{P}}

\newcommand{\D}{\mathcal{D}}
\newcommand{\E}{\mathbb{E}}

\newcommand{\PP}{\mathbb{P}}
\newcommand{\p}{\mathbf{p}}
\newcommand{\loc}{\text{loc}}
\newcommand{\as}{\text{ a.e }}

\title[]{Random Bernoulli measures in a random environment}


\begin{document}

\author{Changhao Chen}
\address{School of Mathematical Sciences, Anhui University, Hefei 230601, China}
\email{chench@ahu.edu.cn}

\author{Zhi Liu}
\address{School of Mathematical Sciences, Anhui University, Hefei 230601, China}
\email{3297152015@qq.com}

\begin{abstract} 
We study random Bernoulli measures $\mu_{\mathbf{p}}$ in a random environment $\mathbf{p}=(p_1,p_2,\ldots)$, where $p_n$ are independent and identically distributed on $(0,1)$. For almost every environment $\mathbf{p}$, we establish several almost sure properties of $\mu_{\mathbf{p}}$: its local dimension, $L^q$ dimensions, Rajchman property, mutual singularity with any fixed measure, and a dichotomy for normal numbers of its typical points.
\end{abstract}

\maketitle

\section{Introduction}

Bernoulli measures are fundamental in probability theory, dynamical systems, and measure theory, and have been extensively studied. In this paper, we investigate a random version of Bernoulli measures, where the parameters are drawn from a random environment. We establish several basic almost sure properties of these random Bernoulli measures. 

To this end, we first recall the definition of the classical Bernoulli measure $\mu_p$ and collect some of its well-known properties concerning local dimension, $L^q$ dimensions, Rajchman property, mutual singularity, and normal numbers. We then present the precise construction of the random environment and state our main results in Theorem~\ref{thm}.

\subsection{Bernoulli measures}
\label{sec:BM}
We begin with the classical setting. For a fixed parameter $p\in(0,1)$, the Bernoulli measure $\mu_p$ on $[0,1)$ is defined as the probability measure satisfying
\[
\mu_p\bigl(I(x_1,\ldots,x_n)\bigr) = p^{\sum_{k=1}^n x_k} (1-p)^{n-\sum_{k=1}^n x_k}
\]
for every $n\in\mathbb{N}$ and every binary word $(x_1,\ldots,x_n)\in\{0,1\}^n$, where
\[
I(x_1,\ldots,x_n) = \left[ \sum_{k=1}^n \frac{x_k}{2^k},\; \sum_{k=1}^n \frac{x_k}{2^k} + \frac{1}{2^n} \right)
\]
is the dyadic interval of length $2^{-n}$. In particular, if $p=1/2$, then the Bernoulli measure $\mu_p$ will be Lebesgue  measure.  For more details and dimension properties of Bernoulli measures, see Heurteaux \cite{H2016}.

\subsection{Terminologies about measures}
Before introducing the random version of Bernoulli measures, we recall several notions that will be used to characterize their properties. Let $\mu$ be a measure on $\R$, and let $B(x, r)$ denote the ball centered at $x$ with radius $r$. The upper local dimension at $x$ is defined by
\[
\overline{\dim}_{\loc}(\mu, x) = \limsup_{r \to 0} \frac{\log \mu(B(x, r))}{\log r},
\]
and the lower local dimension at $x$ is defined by
\[
\underline{\dim}_{\loc}(\mu, x) = \liminf_{r \to 0} \frac{\log \mu(B(x, r))}{\log r}.
\]
If the upper and lower local dimensions are equal, their common value is called the local dimension at $x$, denoted by $\dim_{loc}(\mu, x)$.

Let $\mu$ be a measure on $[0,1)$  and  $q\in \R\setminus\{1\}$.  The upper $L^q$ spectrum and the upper $L^q$ dimension of the measure $\mu$ are respectively defined as 
\[
\overline{\tau}(\mu,q) = \limsup_{n\to\infty} \frac{\log \sum_{I \in \mathcal{D}_n} \mu(I)^q}{\log 2^{-n}}, \quad 
\overline{ D}(\mu,q) = \frac{\overline{\tau} (\mu,q)}{q-1}.
\]
Similarly, we have the lower $L^q$ spectrum $\underline{\tau}(\mu, q)$ and the lower $L^q$ dimension $\underline{ D}(\mu,q)$. If the corresponding upper dimension and lower dimension are equal, their common value are defined respectively as 
\[
\tau(\mu, q) \text{ and } D(\mu, q).
\]
For details of local dimension and $L^q$ dimensions and more discussion, see Falconer \cite[Chapter 17]{Falconer2003}.

Let $\mu$ be a finite Borel measure on $\mathbb{R}^d$. The Fourier transform of $\mu$ at $\xi\in \R^d$ is defined as
\[
\widehat{\mu}(\xi)=\int_{\mathbb{R}} e^{-2\pi i x\xi}\,d\mu(x).
\]
The Fourier transform is a classical concept in harmonic analysis and its applications. More recently, it has gained importance also in geometric measure theory due to its various connections and application in dimension theory, see Mattila \cite{Mattila2015} for more details. We say that the measure $\mu$ is a Rajchman measure if $\widehat{\mu}(\xi) \to 0$ as $ |\xi| \to \infty$. The Rajchman property alone is an important regularity characterisation of the measure being intimately related to various analytic and geometric properties of the measure and its support, see e.g. the survey Lyons \cite{Lyons1995}.

Two Radon measures $\eta$ and $\mu$ are mutually singular, written $\eta \perp \mu$, if there exists a Borel set $B\subseteq \mathbb{R}^n$ such that
\[
\eta(B)=\mu(\mathbb{R}^n\setminus B)=0.
\]
For more details on singularity property of measures, see Mattila \cite[Chapter 2]{Mattila1995}.

A real number $x\in[0,1)$ is called normal if for each integer number $b\ge 2$, every digit in its base-$b$ expansion appears with asymptotic density $1/b$, and every block of digits of length $k$ appears with density $1/b^k$. The Borel normal number theorem  claims that  almost every number $x\in \R$ is normal in the sense of Lebesgue measure. For details and more discussion see Chung \cite[Chapter 5]{Chung}.

We collect some known properties of Bernoulli measure in the following which will be extended to random Bernoulli measures in Theorem \ref{thm}. 
 
\begin{itemize}

\item[(1)] Let $p\in (0,1)$ and $h(p) = -p\log p - (1-p)\log (1-p)$, then  for $\mu_p \as x$,
\[
\dim_{\loc}(\mu_p, x)=h(p).
\]

\item[(2)] Let $p\in (0,1)$ and $q\in \R\setminus \{1\}$, then 
\[
D(\mu, q)= -\frac{1}{q-1} \log ((1-p)^q+p^q).
\]

\item[(3)] If $p\neq 1/2$, then $\mu_p$ is a not a Rajchman measure.

\item[(4)] For any $p\neq q$, the Bernoulli measures $\mu_p, \mu_q$ are mutually singular. 

\item[(5)] If $p=1/2$, then $\mu_{1/2}$ is Lebesgue measure and hence  $\mu_{1/2} \as x $ is normal (Borel normal number theorem).

\item[(6)] If $p\neq 1/2$, then $\mu_p \as x$ is not normal.

\end{itemize}

With these definitions in hand, we now turn to the random setting.

\subsection{Random Bernoulli measures}

We first recall the definition of an inhomogeneous Bernoulli measure. Let $\p = (p_1, p_2, \ldots)$ be a sequence with $p_n \in (0,1)$. Define 
\begin{equation}
\label{eq:mupn}
\mu_{\p, n}\bigl(I(x_1,\ldots,x_n)\bigr) = \prod_{i=1}^n p_i^{x_i} (1-p_i)^{1-x_i},
\end{equation}
where $I(x_1,\ldots,x_n)$ is the dyadic interval of length $2^{-n}$ determined by the binary digits $x_1,\ldots,x_n$ as in Section~\ref{sec:BM}. The sequence $\mu_{\p,n}$ converges weakly to a unique probability measure $\mu_\p$ on $[0,1)$, called the inhomogeneous Bernoulli measure with parameter sequence $\p$.

We now introduce randomness. Let $\nu$ be a Borel probability measure on $(0,1)$. Consider an i.i.d. sequence $\p = (p_1, p_2, \ldots)$ with common distribution $\nu$, i.e., the $p_n$ are independent and $\mathbb{P}(p_n \in A) = \nu(A)$ for each $n\in\mathbb{N}$. The sample space is 
\[
\Omega = (0,1)^{\mathbb{N}} = \{(p_1, p_2, \ldots) : p_k \in (0,1),\; k\in\mathbb{N}\},
\]
equipped with the product probability measure $\PP = \nu^{\mathbb{N}}$. For each $\p\in\Omega$, we obtain an inhomogeneous Bernoulli measure $\mu_\p$. We call $\mu_\p$ a random Bernoulli measure in a random environment $\nu$.

The following theorem summarizes the main results of this paper.

\begin{theorem}\label{thm} 
Using the above notations, we have

\begin{itemize}
\item[(1)]  For $\PP \as  \p\in \Omega$ and $\mu_\p \as  x\in [0,1]$, we have
\[
\dim_{\loc} (\mu_\p, x) =\int h(p) d\nu (p).
\]

\item[(2)] For $\PP \as  \p\in \Omega$ and \textbf{all} $q\in \R\setminus \{1\}$, the $L^q$ dimension of $\mu_\p$ equals 
\[
D(\mu_\p, q)=-\frac{1}{q-1}\int \log ((1-p)^q+p^q) d\nu(p).
\]  

\item[(3)] Suppose that $\nu\neq \delta_{1/2}$, then for $\PP \as  \p\in \Omega$ the measure $\mu_\p$ is a non-Rajchman measure.

\item[(4)] Let $\eta \in \cP([0,1])$ be an arbitrary fixed probability measure. Suppose that $\nu$ is not a Dirac measure, then for $\PP \as  \p\in \Omega$, 
\[
\mu_\p \perp \eta.
\]

\item[(5)] If $\int p d\nu(p)=1/2$, then for $\PP \as  \p\in \Omega$ and $\mu_\p \as  x\in [0,1]$ is normal.

\item[(6)] If $\int p d\nu(p)\neq 1/2$, then for $\PP \as  \p \in \Omega$ and $\mu_\p \as  x\in [0,1]$ is not normal.

\end{itemize}
\end{theorem}

\section{Proofs of main results}

This section is devoted to the proofs of Theorem~\ref{thm}. We will prove each part in a separate subsection.

\subsection{Almost surely local dimension}

We begin with the proof of Theorem~\ref{thm}(1). First, we recall the following version of the law of large numbers, which will be used repeatedly. Let $\{X_n\}$ be a sequence of independent random variables such that $|X_n|\le M$ for all $n$. 
Then almost surely 
\[
\frac{1}{n}\sum_{k=1}^n\bigl(X_k-\mathbb{E}(X_k)\bigr)\rightarrow 0,\text{ as } n\to\infty.
\]

We next introduce the dyadic version of the local dimension of a measure. 
Let $\D_n, n\in \N$ be the $n$-th dyadic partition of the interval $[0,1)$, that is 
\[
\D_n=\left \{ \left [\frac{k}{2^{n}}, \frac{k+1}{2^{n}} \right ): k=0, 1, \ldots, 2^{n}-1 \right \}.
\]
For dyadic expansion  $x=0.x_1x_2\ldots\in[0,1)$, we require that $x_n \neq 1$ for all sufficiently large  $n$. Hence every $x\in [0,1)$ has a unique expansion. Let $I_n(x)=I(x_1, \ldots, x_n)\in \D_n$ be the unique interval which contains $x$. For computing local dimensions of measures, it is often convenient to take dyadic intervals instead of balls. This is due to the following fact. For any measure $\eta $ on $[0,1)$ and $\eta \as x$ one has 
\[
\overline{\dim}_{\loc}(\eta, x)= \limsup_{n\rightarrow \infty} \frac{\log \eta(I_n(x))}{\log 2^{-n}} \text{ and } \underline{\dim}_{\loc}(\eta, x)= \liminf_{n\rightarrow \infty} \frac{\log \eta(I_n(x))}{\log 2^{-n}}.
\]

\begin{lemma}\label{inhomogeneous}
Let $\p=(p_1, p_2, \ldots )\in (0,1)^{\N}$ and $\mu_\p$ be the corresponding  inhomogeneous Bernoulli measure. Then for $\mu_\p \as x$,
\[
\lim_{n\to\infty}\left( \frac{\log\mu_\p(I_n(x))}{\log 2^{-n}} - \frac{1}{n}\sum_{i=1}^n h(p_i) \right) = 0.
\]
\end{lemma}
\begin{proof}
Let $x\in[0,1)$ and let $I_n(x)$ be the $n$-th dyadic interval containing $x$. Then
\[
\begin{aligned}
\frac{\log\mu_\p(I_n(x))}{\log 2^{-n}}
&= -\frac{1}{n}\log\prod_{i=1}^n p_i^{x_i}(1-p_i)^{1-x_i} \\
&= -\frac{1}{n}\sum_{i=1}^n \bigl(x_i\log p_i + (1-x_i)\log(1-p_i)\bigr).
\end{aligned}
\]
Taking expectation with respect to $\mu_\p$, we obtain
\[
\begin{aligned}
\mathbb{E} \left ( x_i\log p_i + (1-x_i)\log(1-p_i) \right )
&= p_i\log p_i + (1-p_i)\log(1-p_i) \\
&= -h(p_i).
\end{aligned}
\]
The law of large numbers implies that for $\mu_\p \as x$, 
\[
\lim_{n\to\infty}\left( \frac{\log\mu_\p(I_n(x))}{\log 2^{-n}} - \frac{1}{n}\sum_{i=1}^n h(p_i) \right) = 0,
\]
which completes the proof.
\end{proof}

\begin{proof}[Proof of Theorem \ref{thm} (1)]
The random variables  $p_1, p_2, \ldots$  are  independent and identically distributed. The law of large numbers implies that for  $\PP \as \p=(p_1, p_2, \ldots)$, 
\[
\frac1n\sum_{i=1}^n h(p_i)\to\mathbb{E}(h(p))=\int h(p)\,d\nu(p), \text{ as } n \rightarrow \infty.
\]
Combining with Lemma \ref{inhomogeneous}, for $\mathbb{P} \as \p$ and $\mu_\p \as x\in[0,1)$, we derive
\[
\frac{\log\mu_\p(I_n(x))}{\log 2^{-n}}\to\int h(p)\,d\nu(p), \text{ as } n \rightarrow \infty,
\]
which finishes the proof.
\end{proof}

\subsection{Almost surely $L^q$ dimensions}

We now prove Theorem~\ref{thm}(2). The following elementary lemma will be useful.

\begin{lemma}
\label{lem:easy}
Let $\p=(p_1, p_2, \ldots)\in (0,1)^{\N}$. Suppose that for any rational number $q$,
\[
\frac{1}{n} \sum_{k=1}^{n}\log ((1-p_k)^{q}+p_k^{q})
\longrightarrow \int \log ((1-p)^q+p^q) d\nu(p), \,\, n \rightarrow \infty.
\]
Then this identity holds for all $q\in \R$.
\end{lemma}

\begin{proof}[Proof of Theorem \ref{thm} (2)]
For $\p=(p_1, p_2, \ldots)\in (0,1)^{\N}$, $q\in \R$ and $n\in \N$, let 
\[
S_n(\mu_\p, q)=\sum_{I\in \D_n} \mu_\p(I)^{q}.
\]
Then for $n>1$, we have 
\[
S_n(\mu_\p, q)=\sum_{J\in \D_{n-1}} \mu_\p(J)^{q}((1-p_n)^{q}+p_n^{q})=\prod_{k=1}^{n}((1-p_k)^{q}+p_k^{q}),
\]
and hence 
\begin{equation}
\label{eq:log}
\frac{1}{n}\log S_n(\mu_\p, q)= \frac{1}{n} \sum_{k=1}^{n}\log ((1-p_k)^{q}+p_k^{q}).
\end{equation} 
Since the rational numbers are countable and the random variables  $p_1, p_2, \ldots$  are  independent and identically distributed, the law of large numbers implies that for $\PP \as \p=(p_1, p_2, \ldots)$ and any rational number $q$, 
\begin{align*}
-\frac{1}{n}\log S_n(\mu_\p, q)&=-\frac{1}{n} \sum_{k=1}^{n}\log ((1-p_k)^{q}+p_k^{q})\\
&\longrightarrow -\int \log ((1-p)^q+p^q) d\nu(p), \,\, n\rightarrow \infty.
\end{align*}
Together with Lemma \ref{lem:easy} and the definition of $D(\mu, q)$, we obtain the desired result.
\end{proof}

\subsection{Almost surely non-Rajchman measure}
Next, we prove Theorem~\ref{thm}(3). It is known that a probability measure $\mu$ on $[0,1]$ is Lebesgue measure if and only if all its non-zero Fourier coefficients vanish. Thus any non-Lebesgue measure has at least one non-zero Fourier coefficient. This qualitative observation, however, is not enough for our purpose. We need a quantitative relation between the Fourier coefficients and the mass discrepancy of $\mu$, which is given by the following lemma.

\begin{lemma}\label{Fourier}
Let \(\mu\in\mathcal{P}([0,1])\), and let \(I,J\subseteq[0,1]\) be intervals with \(|I|=|J|\). 
If \(|\mu(I)-\mu(J)|\ge c>0\), then
\[
\sup_{j\in\mathbb{Z},\,j\neq 0}\bigl|\widehat{\mu}(j)\bigr|\gtrsim c^2.
\]
\end{lemma}

\begin{proof}
By definition of the characteristic function,
\[
\mu(I)-\mu(J)=\int\chi_I(x)\,d\mu(x)-\int\chi_J(x)\,d\mu(x).
\]
Suppose \(0<\varepsilon<\frac{c}{8}\). Let \(I_\varepsilon\) be the interval obtained by enlarging \(I\) by \(\varepsilon\) at both ends, 
and \(J_\varepsilon\) be the interval obtained by shrinking \(J\) by \(\varepsilon\) at both ends.  

Let \(\phi\in C^{\infty}(\mathbb{R})\) with \(\mathrm{spt}\,\phi\subseteq(-1,1)\) and \(\int\phi=1\). Defining
\[
\phi_{\varepsilon}(x)=\frac{1}{\varepsilon}\phi\left(\frac{x}{\varepsilon}\right).
\]
Then $
\mathrm{spt}\,\phi_{\varepsilon}\subseteq(-\varepsilon,\varepsilon)$ and $ \int\phi_{\varepsilon}=1$. 
Observe that 
\[
\chi_I(x)\le\chi_{I_\varepsilon}\ast\phi_\varepsilon(x),\quad \chi_J(x)\ge\chi_{J_\varepsilon}\ast\phi_\varepsilon(x).
\]
It follows that
\[
\begin{aligned}
\int\chi_I\,d\mu-\int\chi_J\,d\mu
&\le\int\chi_{I_\varepsilon}\ast\phi_\varepsilon\,d\mu-\int\chi_{J_\varepsilon}\ast\phi_\varepsilon\,d\mu\\
&=\sum_{n\in\mathbb{Z}}\widehat{\chi_{I_\varepsilon}}(n)\widehat{\phi}(n\varepsilon)\overline{\widehat{\mu}}(n)
-\sum_{n\in\mathbb{Z}}\widehat{\chi_{J_\varepsilon}}(n)\widehat{\phi}(n\varepsilon)\overline{\widehat{\mu}}(n)\\
&=|I_\varepsilon|-|J_\varepsilon|
+\sum_{n\neq 0}\bigl(\widehat{\chi_{I_\varepsilon}}(n)-\widehat{\chi_{J_\varepsilon}}(n)\bigr)\widehat{\phi}(n\varepsilon)\widehat{\mu}(-n)\\
&\lesssim\frac{c}{2}+\sup_{j\neq 0}\bigl|\widehat{\mu}(j)\bigr|\cdot\sum_{n\neq 0}\bigl|\widehat{\phi}(n\varepsilon)\bigr|.
\end{aligned}
\]
By the decay property of Fourier transform of smooth functions, we have
\[
\sum_{n\neq 0}\bigl|\widehat{\phi}(n\varepsilon)\bigr|
\le\sum_{n=1}^\infty\frac{1}{(1+n\varepsilon)^2}
\le\sum_{0<n<1/\varepsilon}\frac{1}{(1+n\varepsilon)^2}
+\sum_{n\ge 1/\varepsilon}\frac{1}{(1+n\varepsilon)^2}
\lesssim\frac{1}{\varepsilon}.
\]
Thus for any \(0<\varepsilon<\frac{c}{8}\),
\[
c\le\mu(I)-\mu(J)\lesssim\frac{c}{2}+\sup_{j\neq 0}\bigl|\widehat{\mu}(j)\bigr|\cdot\frac{1}{\varepsilon}.
\]
By the choice of $\varepsilon$, we have
\[
\sup_{j\in\mathbb{Z},\,j\neq 0}\bigl|\widehat{\mu}(j)\bigr|\gtrsim c^2.
\]
This completes the proof.
\end{proof}

\begin{proof}[Proof of Theorem \ref{thm} (3)]
Since $\nu\neq \delta_{1/2}$, there is an interval $A=[a,b]\subset (0,1)$ such that $1/2\notin A$ and $\nu(A)>0$. Let
\[
\Omega'=\{\p=(p_1, p_2, \ldots): \exists \text{ infinitely many } k \text{ such that } p_k\in A\}.
\]
Since the random variables $p_1, p_2, \ldots $ are independent and indentically distributed, by Borel-Cantelli we obtain that $\PP(\Omega')=1$. We intend to show  that for each $\p\in \Omega'$, 
\[
\widehat{\mu_\p}(\xi)\nrightarrow 0 \text{ as } \xi \rightarrow \infty.
\]

Let $k\in \N$ such that $p_k\in A$. Define a map $f_k\colon [0,1)\to[0,1)$,
\[
f_k(x)=2^kx \mod 1,
\]
which maps every dyadic interval of level \(k\) onto \([0,1)\). This naturally induces the pushforward measure
\[
(f_k)_*\mu_{\p}=\mu_{\p}\circ f_k^{-1}.
\]
Then we have
\[
\left|(f_k)_*\mu_{\p}\Bigl(\bigl(0,\frac{1}{2}\bigr)\Bigr)-(f_k)_*\mu_{\p}\Bigl(\bigl(\frac{1}{2},1\bigr)\Bigr)\right|=|2p_k-1|\ge \min\{|2a-1|, |2b-1|\}>0.
\]
By Lemma \ref{Fourier}, there exists an absolute  constant \(c>0\) such that
\[
\sup_{j\in\mathbb{Z},\,j\neq 0}\bigl|\widehat{(f_k)_*\mu_{\p}}(j)\bigr|\ge c.
\]
Hence there exists \(j_k\in \Z\setminus \{0\}\) satisfying $|\widehat{(f_k)_*\mu_{\p}}(j_k)|\ge c/2$. Observe that for any $j\in \Z$, 
\[
 \widehat{(f_k)_*\mu_{\p}}(j)=\widehat{\mu}_{\p}(2^kj). 
\]
For each $\p\in \Omega'$, we conclude that there are infinitely many $k\in \N, j_k\in \Z\setminus \{0\}$ such that 
\[
|\widehat{\mu}_{\p}(2^k j_k)|\ge c/2.
\]
Thus $\widehat{\mu_{\p}}(\xi)\nrightarrow 0$ as $|\xi|\to\infty$ which finishes the proof.
\end{proof}

\subsection{Almost surely singularity}

We now prove Theorem~\ref{thm}(4). The proof relies on a dyadic version of the density theorem and a convexity argument. Let $\eta, \mu$ be two Radon measures. Recall that for two Radon measures $\eta$ and $\mu$, mutual singularity is characterized by the following density limit: 
\[
\lim_{r\rightarrow 0} \frac{\eta(B(x, r)}{\mu(B(x,r))}. 
\]
In our setting, it is convenient to work with dyadic intervals rather than balls. For $x\in[0,1)$ and $n\in\mathbb{N}$, let $I_n(x)\in\mathcal{D}_n$ be the unique dyadic interval of length $2^{-n}$ containing $x$. The following dyadic version of the density results  will be needed. For the classical version with balls, see \cite[p36, p43]{Mattila1995}.

\begin{lemma}
\label{lem:singular}
Let $\eta, \mu$ be two Radon measures on $[0,1)$, then for $\mu \as x$, 
\[
\lim_{n\rightarrow \infty} \frac{\eta(I_n(x))}{\mu(I_n(x))} \quad \text{exists}.
\]
Moreover, they are mutually singular if and only if for $\mu \as x$, 
\[
\lim_{n\rightarrow \infty}\frac{\eta(I_n(x))}{\mu(I_n(x))}=0.
\]
\end{lemma}

We now establish a convexity estimate that will be used in the proof of singularity.

\begin{lemma}
\label{lem:convex}
Let $\nu\in \cP((0,1))$ such that $\nu$ is not a Dirac measure, then there exists constant $0<\rho<1$ such that 
\[
\left (\int p^{\frac{1}{2}} d\nu(p)\right )^{2} +\left (\int (1-p)^{\frac{1}{2}} d\nu(p) \right )^{2}\le \rho. 
\]
\end{lemma}
\begin{proof}
Since $\nu$ is not a Dirac measure, the functions $p \mapsto p^{1/2}$ and $p \mapsto (1-p)^{1/2}$ are not $\nu$-almost everywhere constant. By the strict convexity of $x \mapsto x^2$, we obtain the strict inequalities
\[
\left (\int p^{\frac{1}{2}} d\nu(p)\right )^{2}<\int p d\nu(p) \text{ and }  \left (\int (1-p)^{\frac{1}{2}} d\nu(p) \right )^{2}<\int (1-p) d\nu(p),
\]
which proves the inequality.
\end{proof}

\begin{proof}[Proof of Theorem \ref{thm} (4)]
By Lemma \ref{lem:singular}, for $\mu_{\p} \as x$,
\[ 
\lim_{n\rightarrow \infty} \frac{\eta(I_n(x))}{\mu_{\p}(I_n(x))} \quad \text{exists}.
\]
Thus it is sufficient to show that 
\begin{equation}
\label{eq:0}
\mathbb{E} \int_{[0,1)} \lim_{n\rightarrow \infty}\left( \frac{\eta(I_n(x))}{\mu_{\p}(I_n(x))} \right)^{1/2} d\mu(x)=0.
\end{equation}
Indeed, this will imply that almost surely for $\mu_\p \as x$,
\[
\lim_{n\rightarrow \infty}\frac{\eta(I_n(x))}{\mu_\p(I_n(x))}=0.
\]
Combining with Lemma \ref{lem:singular} will yields the desired result. We now turn to the proof of identity \eqref{eq:0}. By Fatou's lemma,  the definition of $I_n(x)$ and Cauchy-Schwarz, we derive
\begin{align*}
\mathbb{E} \int_{[0,1)} \lim_{n\rightarrow \infty}\left( \frac{\eta(I_n(x))}{\mu_\p(I_n(x))} \right)^{1/2} d\mu_\p(x) & \le \liminf_{n\rightarrow \infty} \mathbb{E} \int_{[0,1)} \left( \frac{\eta(I_n(x))}{\mu_\p(I_n(x))} \right)^{1/2} d\mu_\p(x)\\
& = \liminf_{n\rightarrow \infty} \mathbb{E} \sum_{I\in \D_n} \eta(I)^{1/2}\mu_\p(I)^{1/2}\\
&\le \liminf_{n\rightarrow \infty} \eta([0,1))^{1/2} \left (\sum_{I\in \D_n} \left (\E(\mu_\p(I)^{1/2})\right)^{2} \right )^{1/2}. 
\end{align*}
We now estimate the term $\E \left (\mu_\p(I)^{1/2}\right )$. For $\p=(p_1, p_2, \ldots)$ and  $I_n(x)=[x_1,\ldots,x_n]$, we have 
\[
\mu_\p(I_n(x))=\prod_{k=1}^{n}(1-p_k)^{1-x_k}p_k^{x_k}.
\]
Since the random variables $p_1, p_2, \ldots $ are independent and indentically distributed, and $x_k= 0$ or $1, k\in \N$, we deduce 
\[
\E \left (\mu_\p(I_n(x))^{1/2}\right )=\prod_{k=1}^{n}\E \left ((1-p_k)^{(1-x_k)/2}p_k^{x_k/2}\right )=\prod_{k=1}^{n} a^{1-x_k}b^{x_k},
\]
where 
\[
a= \E((1-p)^{1/2}) \text{ and } b= \E(p^{1/2}).
\]
Combining with Lemma \ref{lem:convex}, we obtain 
\[
\sum_{I\in \D_n} \left (\E(\mu_\p(I)^{1/2})\right)^{2}=  \sum_{I\in \D_n} \prod_{k=1}^{n} a^{2(1-x_k)}b^{2x_k}=(a^2+b^2)^{n}\le \rho^{n}
\]
where $0<\rho<1$ is constant, and this yields the identity \eqref{eq:0} which finishes the proof. 
\end{proof}

\subsection{Dichotomy property of normal numbers}

Finally, we prove Theorem~\ref{thm}(5) and (6). We begin with a preparatory lemma.

\begin{lemma}\label{regular number}
Suppose that $\int p d\nu(p)=1/2$ and $F\subseteq [0,1]$ with $\mathcal{L}(F)=0$, then almost surely $\mu_\p(F)=0$.
\end{lemma}

\begin{proof}
Let $\varepsilon >0$, then there exists an open set $U\supseteq F$ such that $\mathcal{L}(U)<\varepsilon$. Let $\mu_{\p,n}$ be given by  \eqref{eq:mupn}, then $\mu_{\p,n}$ converges weakly to a measure $\mu_\p$. By Portmanteau theorem \cite[Theorem 1.24]{Mattila1995}, we have 
    \[
    \mu_\p(U) \le \liminf_{n \to \infty} \mu_{\p, n}(U).
    \]
Note that 
\[
\mu_{\p, n}(U) = \int \chi_U(x) \prod_{k=1}^n p_k^{x_k}(1-p_k)^{1-x_k} 2^{n} dx.
\]
Since $\int p d\nu(p)=1/2$ and $\{p_k\}$ are i.i.d. random variables, we obtain 

\begin{align*}
\mathbb{E}\left( \mu_{\p, n}(U) \right) &= \int \chi_U(x) \prod_{k=1}^n \mathbb{E}\left( p_k^{x_k}(1-p_k)^{1-x_k} \right) 2^n dx \\
&= \int \chi_U(x) dx \\
&< \varepsilon.
\end{align*}
Combining with Fatou's lemma,

\[
\mathbb{E}\left( \mu_\p(F) \right) \leq \mathbb{E}\left( \liminf_{n \to \infty} \mu_{\p, n}(U) \right) \leq \liminf_{n \to \infty} \mathbb{E}\left( \mu_{\p, n}(U) \right) \leq \varepsilon.
\]
By the arbitrary choice of $\varepsilon$, we finish the proof.
\end{proof}

\begin{proof}[Proof of Theorem \ref{thm} (5)] Borel  normal number theorem claims that almost every number $x$ is normal in the sense of Lebesgue measure.  Thus we obtain $\mathcal{L}(F) = 0$ where
\[
F = \{ x \in [0,1] : x \text{ is not a normal number} \}.
\]
 Combining with Lemma \ref{regular number}, we derive that $\mu_\p(F) = 0$ for $\PP \as \p$ which implies that  $\mu_p \as x $ is normal.
\end{proof}

\begin{proof}[Proof of Theorem \ref{thm} (6)]
We take the dyadic expansion of each  $x\in [0,1)$. Let $\p=(p_1,p_2,\ldots)$, then the law of large numbers implies that for $\mu_\p \as x=(x_1,x_2,\ldots)\in [0,1)$, 
\begin{equation}
\label{eq,1}
\lim_{n\rightarrow \infty}\frac{1}{n}\sum_{k=1}^n(x_k-p_k) =0.
\end{equation}
Again by law of large numbers, for $\mathbb{P} \as \p=(p_1,p_2,\ldots)$, we obtain
    \[
    \lim_{n\rightarrow \infty} \frac{1}{n}\sum_{k=1}^n p_k = \int p d\nu(p)\neq \frac{1}{2}.
    \]
Combining with \eqref{eq,1}, we conclude that for $\PP \as  \p=(p_1, p_2, \ldots)\in \Omega$ and $\mu_\p \as  x=(x_1,x_2,\ldots)\in [0,1)$,
    \[
   \lim_{n\rightarrow \infty} \frac{1}{n}\sum_{k=1}^n x_k   \neq  \frac{1}{2},
    \]
which implies that $\mu_\p \as x\in [0,1)$ is not normal. 
\end{proof}

\section*{acknowledgement}
This work was supported by the National Natural Science Foundation of China Grant 12101002.

\end{document}